\documentclass[11pt,reqno]{amsart}

\usepackage{amssymb}
\usepackage{amsmath}
\usepackage{amsthm}
\usepackage{amsfonts}
\usepackage{bbm}
\usepackage{a4wide}

\usepackage{bookmark}
\usepackage{hyperref}
\hypersetup{pdfstartview={FitH}}

\newtheorem{theorem}{Theorem}
\newtheorem{lemma}[theorem]{Lemma}
\newtheorem{corollary}[theorem]{Corollary}

\theoremstyle{definition}
\newtheorem{remark}{Remark}

\DeclareFontFamily{U}{mathx}{\hyphenchar\font45}
\DeclareFontShape{U}{mathx}{m}{n}{
<5> <6> <7> <8> <9> <10>
<10.95> <12> <14.4> <17.28> <20.74> <24.88>
mathx10
}{}
\DeclareSymbolFont{mathx}{U}{mathx}{m}{n}
\DeclareFontSubstitution{U}{mathx}{m}{n}
\DeclareMathAccent{\widecheck}{0}{mathx}{"71}

\numberwithin{equation}{section}

\begin{document}

\title{Fourier restriction implies maximal and variational Fourier restriction}

\author{Vjekoslav Kova\v{c}}
\address{Vjekoslav Kova\v{c}, Department of Mathematics, Faculty of Science, University of Zagreb, Bijeni\v{c}ka cesta 30, 10000 Zagreb, Croatia}
\email{vjekovac@math.hr}



\begin{abstract}
We give an abstract argument that an a priori Fourier restriction estimate for a certain choice of exponents automatically implies maximal and variational Fourier restriction estimates. These, in turn, provide pointwise and quantitative interpretations of restriction of the Fourier transform to the hypersurface in question.
\end{abstract}

\maketitle


\section{Introduction and statement of the results}
The Fourier restriction problem, in its original formulation, starts with a complex-valued function $f$ in the space $\textup{L}^p(\mathbb{R}^d)$ for an appropriate exponent $p\in(1,2]$ and attempts to give a meaning to the restriction of its Fourier transform $\widehat{f}$ to a given hypersurface $S\subset\mathbb{R}^d$.
This task is performed in an operator-flavored fashion, by establishing an a priori estimate for Schwartz functions, namely
\begin{equation}\label{eq:restr}
\big\|\widehat{f}\,\big\|_{\textup{L}^q(S,\sigma)} \leq C_{\textup{restr}} \|f\|_{\textup{L}^p(\mathbb{R}^d)}
\end{equation}
with a finite constant $C_{\textup{restr}}$ and for some exponent $q\in[1,\infty]$. Here $\sigma$ is the surface measure on $S$, appropriately weighted if necessary.
A lot of research has been done on determining the largest possible range of exponents $(p,q)$ in inequality \eqref{eq:restr} for a fixed hypersurface $S$. The program was initiated by Stein; see the first lemma in Section IV of Fefferman's paper \cite{F70:conv}. Today it employs a variety of modern analytical tools, but the full range of exponents remains largely unsolved in dimensions $d\geq 3$. The reader can consult the expository text by Tao \cite{T04:restr} for an overview of developments prior to 2004, and the papers by Bourgain and Guth \cite{BG11:restr}, Guth \cite{G16:restr1,G16:restr2}, and Hickman and Rogers \cite{HR18:restr} for just some of the numerous recent results.

In the following remark we only recall three very classical results in the restriction theory, as we will refer to them later.

\begin{remark}\label{rem:classical}
(a) Three examples of surfaces deserve special attention: the \emph{sphere} $\mathbb{S}^{d-1}\subset\mathbb{R}^d$ with its surface measure $\sigma$, the \emph{whole paraboloid} in $\mathbb{R}^d$,
\[ \Big\{ \xi = \Big(\eta,\frac{1}{2}|\eta|^2\Big) \,:\, \eta\in\mathbb{R}^{d-1} \Big\}, \]
with $\textup{d}\sigma(\xi) = \textup{d}\eta$, and the \emph{whole cone} in $\mathbb{R}^d$,
\[ \big\{ \xi = \big(\eta,|\eta|\big) \,:\, \eta\in\mathbb{R}^{d-1} \big\}, \]
with $\textup{d}\sigma(\xi) = \textup{d}\eta / |\xi|$.
Here $|v|$ denotes the Euclidean length of a vector $v$.
The famous \emph{restriction conjecture} claims that estimate \eqref{eq:restr} holds in the cases of the sphere and the whole paraboloid whenever
\begin{equation}\label{eq:rangesphere}
p<\frac{2d}{d+1},\quad q=\frac{(d-1)p'}{d+1}
\end{equation}
and in the case of the whole cone whenever
\begin{equation}\label{eq:rangecone}
p<\frac{2(d-1)}{d},\quad q=\frac{(d-2)p'}{d}.
\end{equation}
Here and in what follows $p'$ denotes the conjugated exponent of $p$.
Indeed, the so-called Knapp examples (see \cite{T04:restr}) show that conditions \eqref{eq:rangesphere} and \eqref{eq:rangecone} are necessary for estimate \eqref{eq:restr} to hold in the cases of the whole paraboloid and the whole cone respectively. The same reasoning applied to the sphere shows that \eqref{eq:restr} is possible only when $p<2d/(d+1)$ and $q\leq(d-1)p'/(d+1)$, which is simply range \eqref{eq:rangesphere} trivially extended by lowering the exponent $q$.

(b) The restriction problem is essentially solved in $d=2$ dimensions. Estimates \eqref{eq:restr} for compact $\textup{C}^2$ curves $S$ with nonnegative signed curvature and when $\sigma$ is an appropriately weighted arclength measure (i.e., it is the so-called affine arclength measure) were confirmed in the largest possible range, $p<4/3$ and $q\leq p'/3$, by Carleson and Sj\"{o}lin \cite{CS72:2d} and Sj\"{o}lin \cite{S74:2d}.

(c) The $T^\ast T$ method is efficient in the case $q=2$, even in higher dimensions. For $d\geq 3$ and the sphere $\mathbb{S}^{d-1}$ the Tomas--Stein theorem \cite{T75:restr} gives estimate \eqref{eq:restr} for
\begin{equation}\label{eq:rangets}
p\leq\frac{2(d+1)}{d+3},\quad q=2.
\end{equation}
This range of exponent $p$ is optimal with $q$ fixed.
\end{remark}

In this note we rather concentrate on another type of questions, in the direction of the maximal Fourier restriction theory, proposed recently by M\"{u}ller, Ricci, and Wright \cite{MRW16:maxrestr}.
Suppose that $\mu$ is a complex measure on Borel subsets of $\mathbb{R}^d$. (Its finiteness is included in the definition of complex measures.) It will serve as an averaging measure. For $t\in(0,\infty)$ let $\mu_t$ be another complex measure defined by the formula
\[ \mu_t(E) := \mu(t^{-1}E) \]
for each Borel set $E\subseteq\mathbb{R}^d$.
If $\mu$ is absolutely continuous with respect to the Lebesgue measure and its density is $\chi\in\textup{L}^1(\mathbb{R}^d)$, i.e.~$\textup{d}\mu(x)=\chi(x)\textup{d}x$, then $\mu_t$ has density $\chi_t$, where
\[ \chi_t(x) := \frac{1}{t^d} \chi\Big(\frac{x}{t}\Big). \]
It is natural to consider convolution-type averages $\widehat{f}\ast\mu_t$ (or $\widehat{f}\ast\chi_t$) of the Fourier transform of $f$ with respect to the averaging measure $\mu$ (or the function $\chi$).

Before we formulate our results, let us agree on some notation. Let $A,B\colon X\to[0,\infty)$ be two functions defined on the same set and let $P$ be some set of parameters. If there exists a constant $C_P\in[0,\infty)$, depending only on the parameters from $P$, such that $A(x)\leq C_P B(x)$ for each $x\in X$, then we write $A(x)\lesssim_P B(x)$. When the set of parameters $P$ is empty, we simply write $A(x)\lesssim B(x)$.

As the first new result we will show that estimate \eqref{eq:restr}, under reasonable conditions on $p$, $q$, and $\mu$, automatically implies its variant for maximal averages of $\widehat{f}$ .

\begin{theorem}\label{thm:theorem1}
Let $S$ be a Borel subset of $\mathbb{R}^d$ and $\sigma$ a measure on Borel subsets of $S$. Let $\mu$ be a complex-valued measure on Borel subsets of $\mathbb{R}^d$ such that its Fourier transform $\widehat{\mu}$ is a $\textup{C}^\infty$ function satisfying
\begin{equation}\label{eq:mudecay}
|\nabla\widehat{\mu}(x)| \leq D (1+|x|)^{-1-\eta}
\end{equation}
for each $x\in\mathbb{R}^d$ with a constant $D\in[0,\infty)$ and a parameter $\eta\in(0,\infty)$. Finally, take exponents $p\in[1,2]$ and $q\in(1,\infty)$ such that $p<q$.
Suppose that \eqref{eq:restr} holds for every Schwartz function $f$ on $\mathbb{R}^d$, with some constant $C_{\textup{restr}}\in[0,\infty)$.
Then we also have
\begin{equation}\label{eq:maxrestr}
\Big\| \sup_{t\in(0,\infty)} \big| \widehat{f}\ast\mu_t \big| \Big\|_{\textup{L}^q(S,\sigma)}
\lesssim_{p,q,\eta} C_{\textup{restr}} D \|f\|_{\textup{L}^p(\mathbb{R}^d)}
\end{equation}
for every Schwartz function $f$.
\end{theorem}

Even though the condition $p<q$ is not always present in the literature on restriction estimates \eqref{eq:restr}, it is still not very restrictive.
For instance, both ranges \eqref{eq:rangesphere} and \eqref{eq:rangecone} fall inside the region $p<q$, so we immediately get the following conditional result.

\begin{corollary}
If restriction estimate \eqref{eq:restr} holds for the sphere or the whole paraboloid and a pair of exponents satisfying \eqref{eq:rangesphere}, or for the whole cone and a pair of exponents satisfying \eqref{eq:rangecone}, then it automatically upgrades to its maximal variant \eqref{eq:maxrestr}, where $\mu$ is as in Theorem~\ref{thm:theorem1}.
\end{corollary}

Two instances of the maximal Fourier restriction estimate \eqref{eq:maxrestr} have appeared in the literature so far.

First, M\"{u}ller, Ricci, and Wright \cite{MRW16:maxrestr} established the maximal Fourier restriction inequality when $S$ is a compact $\textup{C}^2$ curve with nonnegative signed curvature, $\sigma$ is its affine arclength measure, $\textup{d}\mu(x)=\chi(x)\textup{d}x$, and $\chi$ is a Schwartz function, while the exponents $p,q$ are taken from the largest possible range for \eqref{eq:restr}, which is $p<4/3$, $q\leq p'/3$. In other words, they established a maximal variant of the result from part (b) of Remark~\ref{rem:classical}. By applying Theorem~\ref{thm:theorem1} with $p<4/3$ and $q=p'/3>p$ this result becomes a consequence of papers \cite{CS72:2d} and \cite{S74:2d}.

Second, Vitturi \cite{V17:maxrestr} established \eqref{eq:maxrestr} when $S$ is the sphere $\mathbb{S}^{d-1}$ in $\mathbb{R}^d$, $\sigma$ is the surface measure, $p\leq 4/3$, and $q\leq (d-1)p'/(d+1)$. Combining Theorem~\ref{thm:theorem1} with the Tomas--Stein theorem \cite{T75:restr} and observing that the exponents from \eqref{eq:rangets} still satisfy $p<q$, we can strengthen the result from \cite{V17:maxrestr} in the most interesting case $q=2$.

\begin{corollary}\label{cor:tomstein}
Let $S=\mathbb{S}^{d-1}$ be the sphere in $\mathbb{R}^d$, let $\sigma$ be the surface measure on $\mathbb{S}^{d-1}$, and let $\mu$ be as in Theorem~\ref{thm:theorem1}. The maximal estimate \eqref{eq:maxrestr} holds with $p\leq2(d+1)/(d+3)$ and $q=2$.
\end{corollary}

It is worth mentioning that the method used in \cite{MRW16:maxrestr} runs into problems in dimensions $d\geq 3$ and it is not clear if Corollary~\ref{cor:tomstein} can be deduced in a similar way. On the other hand, the paper \cite{V17:maxrestr} relies on the trick of expanding out the $\textup{L}^4$-norm of the extension operator, so it can only handle a strict subrange of the Tomas--Stein range for $d\geq 4$, i.e., it stays away from the endpoint $p=2(d+1)/(d+3)$. Finally, the most classical $T^\ast T$ trick is not efficient in the context of the maximal Fourier restriction and, just by itself, it certainly cannot prove Corollary~\ref{cor:tomstein}.

With just a bit more effort we can strengthen Theorem~\ref{thm:theorem1} to variational estimates of the averages $\widehat{f}\ast\mu_t$. Let us only remark that the variation norms were introduced to harmonic analysis by Bourgain \cite{B89:pt} in order to prove a.e.~convergence when there is no obvious dense class on which convergence holds.

\begin{theorem}\label{thm:theorem2}
Let $S$, $\sigma$, $D$, $\eta$, $\mu$, $p$, and $q$ be as in Theorem~\ref{thm:theorem1}, and additionally take $\varrho\in(p,\infty)$.
Suppose that \eqref{eq:restr} holds for every Schwartz function $f$ on $\mathbb{R}^d$, with some constant $C_{\textup{restr}}\in[0,\infty)$.
Then we also have
\begin{equation}\label{eq:varrestr}
\bigg\| \sup_{\substack{m\in\mathbb{N}\\ 0<t_0<t_1<\cdots<t_m}} \Big( \sum_{j=1}^{m} \big| \widehat{f}\ast\mu_{t_{j}} - \widehat{f}\ast\mu_{t_{j-1}} \big|^{\varrho} \Big)^{1/\varrho} \bigg\|_{\textup{L}^q(S,\sigma)}
\lesssim_{p,q,\varrho,\eta} C_{\textup{restr}} D \|f\|_{\textup{L}^p(\mathbb{R}^d)}
\end{equation}
for every Schwartz function $f$.
\end{theorem}

A particular case of estimate \eqref{eq:varrestr} when $S$ is the sphere $\mathbb{S}^2$ in $\mathbb{R}^3$, $\textup{d}\mu(x)=\chi(x)\textup{d}x$, $\chi$ is either a Schwartz function or the characteristic function of the unit ball, $p=4/3$, $q=2$, and $\varrho>2$ was established by Oliveira e Silva and the author in \cite{KOeS18:varrestr}, with a clear extension to higher dimensions and the same range of exponents as in \cite{V17:maxrestr}. Theorem~\ref{thm:theorem2} implies this result because the corresponding a priori estimates \eqref{eq:restr} fall inside the range of the Tomas--Stein restriction theorem \cite{T75:restr}; see the comments preceding Corollary~\ref{cor:tomstein}.

Throughout the paper we will be focused on proving Theorem~\ref{thm:theorem2}, while Theorem~\ref{thm:theorem1} will be easily reduced to it.
It is curious to remark that, even if one is primarily interested in the maximal Fourier restriction \eqref{eq:maxrestr}, it is still more in the spirit of our approach to begin by showing the variational estimate \eqref{eq:varrestr}.

A source of motivation for the maximal Fourier restriction, originating in \cite{MRW16:maxrestr}, is that the maximal inequality \eqref{eq:maxrestr} gives a pointwise meaning to the restriction $\widehat{f}\vert_S$. When we combine it with obvious pointwise convergence on a dense subset of $\textup{L}^p(\mathbb{R}^d)$, we can conclude that the limit of $(\widehat{f}\ast\mu_\varepsilon)(\xi)$ as $\varepsilon\to0^+$ exists at $\sigma$-almost every point $\xi\in S$. Moreover, the variational estimate \eqref{eq:varrestr}, when extended by density arguments to an arbitrary function $f\in\textup{L}^p(\mathbb{R}^d)$, quantifies this convergence.

Two remarks should further clarify the applicability of Theorems~\ref{thm:theorem1} and \ref{thm:theorem2}.

\begin{remark}\label{rem:remark1}
Let us comment on the possible choices of the averaging measure $\mu$.
When $\textup{d}\mu(x)=\chi(x)\textup{d}x$ for some $\chi\in\textup{L}^1(\mathbb{R}^d)$, one can take $\chi$ to be a Schwartz function or the characteristic function of the unit ball in $\mathbb{R}^d$. The assumptions on $\mu$ from Theorems~\ref{thm:theorem1} and \ref{thm:theorem2} are clearly satisfied in the first case. In the second case,
\[ \widehat{\chi}(x) = \frac{J_{d/2}(2\pi |x|)}{|x|^{d/2}} \]
and
\[ \nabla\widehat{\chi}(x) = \Big( \frac{\pi J_{d/2-1}(2\pi |x|)}{|x|^{d/2+1}} - \frac{\pi J_{d/2+1}(2\pi |x|)}{|x|^{d/2+1}} - \frac{d J_{d/2}(2\pi |x|)}{2|x|^{d/2+2}} \Big) x, \]
where $J_\alpha$ denotes the Bessel function of the first kind with parameter $\alpha$; see \cite{AS92:hmf} or do the computation in Mathematica \cite{W18:math}. Together with the asymptotics $J_{\alpha}(2\pi|x|)=O(|x|^{-1/2})$ as $|x|\to\infty$, this gives $|\nabla\widehat{\chi}(x)|=O(|x|^{-(d+1)/2})$ as $|x|\to\infty$, so condition \eqref{eq:mudecay} is satisfied with $\eta=(d-1)/2>0$ in all dimensions $d\geq 2$.
If $B(\omega,\varepsilon)$ denotes the ball centered at $\omega$ with radius $\varepsilon$, then $\widehat{f}\ast\mu_\varepsilon$ is a constant multiple of
\[ \frac{1}{|B(\omega,\varepsilon)|} \int_{B(\omega,\varepsilon)} \widehat{f}(\xi) \,\textup{d}\xi. \]
Under the hypotheses of Theorem~\ref{thm:theorem1} these averages converge as $\varepsilon\to0^+$ for $\sigma$-almost every point $\omega\in S$ whenever $f\in\textup{L}^p(\mathbb{R}^d)$. This clearly follows from the maximal estimate \eqref{eq:maxrestr} and pointwise convergence on the dense class consisting of Schwartz functions $f$. Theorem~\ref{thm:theorem2} then gives a quantitative variant of this convergence result.

Moreover, we can, for instance, take $\mu$ to be the surface measure on the unit sphere $\mathbb{S}^{d-1}\subset\mathbb{R}^d$ when $d\geq 4$. This time
\[ \widehat{\mu}(x) = \frac{2\pi J_{d/2-1}(2\pi |x|)}{|x|^{d/2-1}} \]
and the same computation as before (only lowering the parameters of the Bessel functions by $1$) gives $|\nabla\widehat{\mu}(x)|=O(|x|^{-(d-1)/2})$ as $|x|\to\infty$, so condition \eqref{eq:mudecay} is satisfied with $\eta=(d-3)/2>0$ in dimensions $d\geq 4$.
Note that $\widehat{f}\ast\mu_\varepsilon$ is a constant multiple of the spherical means of $\widehat{f}$,
\[ \frac{1}{\mu(\mathbb{S}^{d-1})} \int_{\mathbb{S}^{d-1}} \widehat{f}(\omega+\varepsilon\xi) \,\textup{d}\mu(\xi), \]
which are more singular than averages over balls. Theorems~\ref{thm:theorem1} and \ref{thm:theorem2} imply that the spherical means of $\widehat{f}$ for $f\in\textup{L}^p(\mathbb{R}^d)$ also converge as $\mathbb{Q}\ni\varepsilon\to0^+$ for $\sigma$-almost every point $\omega\in S$, as long as $d\geq 4$ and \eqref{eq:restr} holds with some exponent $q\in(p,\infty)$.
\end{remark}

\begin{remark}\label{rem:remark4}
Suppose that $p\in[1,2]$ and $q\in(1,\infty)$ are exponents such that $p<q$ and that \eqref{eq:restr} holds. Denote $\widetilde{p}=2p/(p+1)$ and $\widetilde{q}=2q$.
By taking $h$ to be $f$ convolved with its reflection about the origin, observing $\widehat{h}=|\widehat{f}|^2$, using Young's inequality for convolutions, and applying Theorem~\ref{thm:theorem1} with $h$ in the place of $f$, we get
\[ \Big\| \sup_{t\in(0,\infty)} \Big( \frac{1}{|B(\omega,t)|} \int_{B(\omega,t)} \big|\widehat{f}(\xi)\big|^2 \,\textup{d}\xi \Big)^{1/2} \Big\|_{\textup{L}^{\widetilde{q}}_{\omega}(S,\sigma)}
\lesssim_{S,\sigma,\widetilde{p},\widetilde{q}} \|f\|_{\textup{L}^{\widetilde{p}}(\mathbb{R}^d)} \]
and consequently also
\[ \Big\| \sup_{t\in(0,\infty)} \frac{1}{|B(\omega,t)|} \int_{B(\omega,t)} \big|\widehat{f}(\xi)\big| \,\textup{d}\xi \Big\|_{\textup{L}^{\widetilde{q}}_{\omega}(S,\sigma)}
\lesssim_{S,\sigma,\widetilde{p},\widetilde{q}} \|f\|_{\textup{L}^{\widetilde{p}}(\mathbb{R}^d)} \]
for each $f\in\textup{L}^{\widetilde{p}}(\mathbb{R}^d)$; recall the possible choices for $\mu$ from Remark~\ref{rem:remark1}. This trick was employed by M\"{u}ller, Ricci, and Wright in \cite{MRW16:maxrestr} in order to show that, for $\widetilde{p}<8/7$ and $f\in\textup{L}^{\widetilde{p}}(\mathbb{R}^2)$, almost every point (with respect to the arclength measure) at which the curvature of the curve $S$ does not vanish is a Lebesgue point for $\widehat{f}$.
Vitturi \cite{V17:maxrestr} repeated this argument in $d\geq 3$ dimensions for the spheres $\mathbb{S}^{d-1}\subset\mathbb{R}^d$ and functions $f\in\textup{L}^{\widetilde{p}}(\mathbb{R}^2)$ in the range $\widetilde{p}\leq 8/7$. Therefore, Theorem~\ref{thm:theorem1} also reproves and generalizes these two results. Interestingly, Ramos \cite{R18:maxrestr} adjusted the techniques from \cite{MRW16:maxrestr} and \cite{V17:maxrestr} in order to extend these results to the ranges $\widetilde{p}<4/3$ and $\widetilde{p}\leq 4/3$ respectively. Theorem~\ref{thm:theorem1} and this remark do not capture his extension.
\end{remark}

Let us say a few words on the scheme of the proof of Theorems~\ref{thm:theorem1} and \ref{thm:theorem2}. Instead of starting with a general averaging measure $\mu$, we begin by considering the case when the averages are convolutions with Schwartz functions that have compact frequency supports. Then we split the proof of estimate \eqref{eq:varrestr} into discussions of long and short variations, as explained, for instance, in the classical paper by Jones, Seeger, and Wright \cite{JSW08:var}. The part dealing with long variations relies simply on the abstract maximal principle of Christ and Kiselev \cite{CK01:max}, also generalized to variational estimates by Oberlin, Seeger, Tao, Thiele, and Wright \cite{OSTTW12:varcar} in a very similar context. The part dealing with short variations is even easier, as it is reduced to boundedness of a certain $\ell^p$-square function in a trivial off-diagonal range. Finally, for a general $\mu$ satisfying the above conditions we perform a smooth Littlewood-Paley-type partition of unity.

It is worth emphasizing that our approach is quite abstract and does not use any geometrical properties of the hypersurface or any properties of the restriction operator apart from the basic symmetries of the Fourier transform. In comparison, M\"{u}ller, Ricci, and Wright \cite{MRW16:maxrestr} rely in part on the approach to restriction estimates by Carleson and Sj\"{o}lin \cite{CS72:2d} and Zygmund \cite{Z74:2d}. Also, the paper by Oliveira e Silva and the present author \cite{KOeS18:varrestr} applies the non-oscillatory reformulation of a particular restriction estimate due to Vitturi \cite{V17:maxrestr}, while variational estimates for ordinary convolution-type averages (i.e.\@ a result by Bourgain \cite{B89:pt}) are still used as a black box. In this paper we need to be more careful when handling the variational estimates in order to guarantee disjointness of supports when performing the inductive step of the Christ--Kiselev argument.

As we have already remarked, this note reproves the main results of the papers \cite{KOeS18:varrestr,MRW16:maxrestr,V17:maxrestr} by recognizing them as direct consequences of the a priori restriction estimate \eqref{eq:restr}. However, our approach cannot entirely replace the techniques from these papers. For instance, M\"{u}ller, Ricci, and Wright \cite{MRW16:maxrestr} in their proof also establish estimates for a certain two-parameter maximal function of the Fourier transform, and it is not clear if these estimates can also be deduced from \eqref{eq:restr} used as a black box. Moreover, Ramos conveniently modified the approaches from \cite{MRW16:maxrestr,V17:maxrestr} in the aforementioned paper \cite{R18:maxrestr} (see Remark~\ref{rem:remark4}), but Theorem~\ref{thm:theorem1} does not reprove his results. The author hopes that this note will hint on how to formulate nontrivial problems and interesting future results in the theory of the maximal Fourier restriction.

\section{Notation}
We begin by explaining the notation used throughout the paper. The letters $\mathbb{N}$, $\mathbb{Z}$, and $\mathbb{R}$ respectively denote the sets of positive integers, all integers, and real numbers. The indicator function of any set $E$ will be denoted by $\mathbbm{1}_E$. The standard scalar product of vectors $x,y\in\mathbb{R}^d$ will be written as $x\cdot y$, while the Euclidean length of $x\in\mathbb{R}^d$ will be written as $|x|$.
Whenever the measure is not specified, it is understood that the integration is performed with respect to the Lebesgue measure. The same applies to the Lebesgue spaces $\textup{L}^p(\mathbb{R}^d)$.
The space of complex Schwartz functions on $\mathbb{R}^d$ is denoted by $\mathcal{S}(\mathbb{R}^d)$.

The Fourier transform of a complex measure $\mu$ on Borel subsets of $\mathbb{R}^d$ is a bounded continuous function defined as
\[ \widehat{\mu}(y) := \int_{\mathbb{R}^d} e^{-2\pi i x\cdot y} \,\textup{d}\mu(x). \]
When $\textup{d}\mu(x)=\chi(x)\textup{d}x$ for some $\chi\in\textup{L}^1(\mathbb{R}^d)$, this definition specializes to
\[ \widehat{\chi}(y) = \int_{\mathbb{R}^d} e^{-2\pi i x\cdot y} \chi(x) \,\textup{d}x. \]
Inverse Fourier transforms of $\mu$ and $\chi$ are simply given by $\widecheck{\mu}(x)=\widehat{\mu}(-x)$ and $\widecheck{\chi}(x)=\widehat{\chi}(-x)$.
Let us also recall that, for $q\in[1,\infty]$, the convolution of a function $g\in\textup{L}^q(\mathbb{R}^d)$ with a complex Borel measure $\mu$ is defined as
\[ (g\ast\mu)(x) := \int_{\mathbb{R}^d} g(x-y) \,\textup{d}\mu(y) \]
and that this function belongs to $\textup{L}^q(\mathbb{R}^d)$ again. For any $f\in\textup{L}^p(\mathbb{R}^d)$ with $p\in[1,2]$ basic properties of the Fourier transform give the formula $\widehat{f}\ast\mu = \widehat{f\widecheck{\mu}}$ and the same is true when $\mu$ is replaced with a function $\chi\in\textup{L}^1(\mathbb{R}^d)$.
Also recall the notation $\mu_t$ and $\chi_t$ explained in the introduction and note that
\[ \widehat{\mu_t}(x) = \widehat{\mu}(tx),\quad \widehat{\chi_t}(x) = \widehat{\chi}(tx). \]

Throughout the rest of the paper we fix the set $S$ and the measure $\sigma$.
We assume that $p\in[1,2]$ and $q\in(p,\infty)$ are such that the estimate \eqref{eq:restr} holds for all functions $f\in\mathcal{S}(\mathbb{R}^d)$ with a finite constant $C_{\textup{restr}}$.
All results will be stated and proved conditionally on that assumption.
We also fix $\varrho\in(p,\infty)$.

\begin{remark}\label{rem:extend}
A standard approximation argument immediately extends \eqref{eq:restr} to functions $f\in\textup{L}^1(\mathbb{R}^d)\cap\textup{L}^p(\mathbb{R}^d)$.
Indeed, for any such function $f$ we can find a sequence $(f_n)_{n=1}^{\infty}$ of Schwartz functions that converges to $f$, both in the $\textup{L}^1$-norm and in the $\textup{L}^p$-norm.
Then $(\widehat{f_n})_{n=1}^{\infty}$ converges pointwise to $\widehat{f}$, so Fatou's lemma combined with \eqref{eq:restr} gives
\[ \big\|\widehat{f}\,\big\|_{\textup{L}^q(S,\sigma)} \leq \liminf_{n\to\infty}\big\|\widehat{f_n}\,\big\|_{\textup{L}^q(S,\sigma)}
\leq C_{\textup{restr}} \lim_{n\to\infty}\|f_n\|_{\textup{L}^p(\mathbb{R}^d)} = C_{\textup{restr}} \|f\|_{\textup{L}^p(\mathbb{R}^d)}. \]
\end{remark}

\section{Variationally truncated Fourier transform}
The following lemma is the main ingredient in the proof of Theorems~\ref{thm:theorem1} and \ref{thm:theorem2}.

\begin{lemma}\label{lm:cklemma}
Suppose that $(\Psi_k)_{k\in\mathbb{Z}}$ are functions in $\mathcal{S}(\mathbb{R}^d)$ with mutually disjoint supports.
\begin{itemize}
\item[(a)] For any $M,N\in\mathbb{Z}$ such that $M\leq N$ and for each function $f\in\mathcal{S}(\mathbb{R}^d)$ we have
\begin{equation}\label{eq:induct1}
\Bigg\| \max_{\substack{n\in\mathbb{Z}\\ M\leq n\leq N}} \Big| \sum_{k=M}^{n} \widehat{f \Psi_k} \Big| \Bigg\|_{\textup{L}^q(S,\sigma)}
\lesssim_{p,q} C_{\textup{restr}} \bigg\| |f| \sum_{k=M}^{N} |\Psi_k| \bigg\|_{\textup{L}^p(\mathbb{R}^d)}.
\end{equation}
\item[(b)] For any $M,N\in\mathbb{Z}$ such that $M<N$ and for each function $f\in\mathcal{S}(\mathbb{R}^d)$ we have
\begin{equation}\label{eq:induct2}
\Bigg\| \sup_{\substack{m\in\mathbb{N}\\ k_0,\ldots,k_m\in\mathbb{Z} \\ M\leq k_0<\cdots<k_m\leq N}} \bigg( \sum_{j=1}^{m} \Big| \sum_{k=k_{j-1}+1}^{k_j} \widehat{f \Psi_k} \Big|^{\varrho} \bigg)^{1/\varrho} \Bigg\|_{\textup{L}^q(S,\sigma)}
\lesssim_{p,q,\varrho} C_{\textup{restr}} \bigg\| |f| \sum_{k=M+1}^{N} |\Psi_k| \bigg\|_{\textup{L}^p(\mathbb{R}^d)}.
\end{equation}
\end{itemize}
\end{lemma}

Such results have become standard in the literature since Christ and Kiselev introduced their general technique in \cite{CK01:max}.
We include a proof of Lemma~\ref{lm:cklemma} for completeness.
Strictly speaking, only part (b) will be needed in the following sections, but part (a) is used in the proof of part (b).

\begin{proof}[Proof of Lemma~\ref{lm:cklemma}]
(a) We closely follow the simplified proof of the maximal Christ--Kiselev lemma by Tao~\cite{T06:lecnotes}.

Regard $M$ and $N$ as fixed. Let us linearize the maximum on the left hand side of \eqref{eq:induct1} by setting $S_n$, for each integer $M\leq n\leq N$, to be the set of all points $\xi\in S$ such that $\big| \sum_{k=M}^{n'} \big(\widehat{f \Psi_k}\big)(\xi) \big|$ attains its first maximum over all integers $M\leq n'\leq N$ precisely at $n'=n$. That way $S$ becomes the disjoint union of $S_M$, \ldots, $S_N$.

For any integers $M',N'\in\mathbb{Z}$ such that $M\leq M'\leq N'\leq N$ we prove the estimate
\begin{equation}\label{eq:induct3}
\Bigg\| \sum_{M'\leq k\leq n\leq N'} \mathbbm{1}_{S_n} \widehat{f \Psi_k} \Bigg\|_{\textup{L}^q(S,\sigma)}
\leq A_{p,q} C_{\textup{restr}} \bigg\| |f| \sum_{M'\leq k\leq N'} |\Psi_k| \bigg\|_{\textup{L}^p(\mathbb{R}^d)}
\end{equation}
with the constant given as
\[ A_{p,q} := \frac{1}{1-2^{1/q-1/p}} \]
by the mathematical induction on the difference $N'-M'\in\{0,1,\ldots,N-M\}$. Observe that the induction will terminate with $M'=M$, $N'=N$, when \eqref{eq:induct3} becomes \eqref{eq:induct1}.

The induction basis $N'-M'=0$ is trivial, since then \eqref{eq:induct3} is follows from \eqref{eq:restr} applied to the function $f\Psi_{N'}$:
\[ \big\| \mathbbm{1}_{S_{N'}} \widehat{f \Psi_{N'}} \big\|_{\textup{L}^q(S,\sigma)}
\leq \big\| \widehat{f \Psi_{N'}} \big\|_{\textup{L}^q(S,\sigma)}
\leq C_{\textup{restr}} \| f\Psi_{N'} \|_{\textup{L}^p(\mathbb{R}^d)}
\leq A_{p,q} C_{\textup{restr}} \| f\Psi_{N'} \|_{\textup{L}^p(\mathbb{R}^d)}. \]
For the induction step we take $M',N'$ with $N'-M'\geq 1$ and suppose that \eqref{eq:induct3} holds for all pairs of cutoff integers with difference strictly smaller than $N'-M'$.
Let $M'\leq K\leq N'$ be the smallest index such that
\[ \bigg\| |f| \sum_{M'\leq k\leq K} |\Psi_k| \bigg\|_{\textup{L}^p(\mathbb{R}^d)}^p > \frac{1}{2} \bigg\| |f| \sum_{M'\leq k\leq N'} |\Psi_k| \bigg\|_{\textup{L}^p(\mathbb{R}^d)}^p. \]
We apply the induction hypothesis to the functions $\Psi_k$ from the ranges $M'\leq k\leq K-1$ and $K+1\leq k\leq N'$ respectively, obtaining
\begin{align}
& \Bigg\| \sum_{M'\leq k\leq n\leq K-1} \mathbbm{1}_{S_n} \widehat{f \Psi_k} \Bigg\|_{\textup{L}^q(S,\sigma)}
\leq A_{p,q} C_{\textup{restr}} \bigg\| |f| \sum_{M'\leq k\leq K-1} |\Psi_k| \bigg\|_{\textup{L}^p(\mathbb{R}^d)}, \label{eq:indauxa1} \\
& \Bigg\| \sum_{K+1 \leq k\leq n\leq N'} \mathbbm{1}_{S_n} \widehat{f \Psi_k} \Bigg\|_{\textup{L}^q(S,\sigma)}
\leq A_{p,q} C_{\textup{restr}} \bigg\| |f| \sum_{K+1\leq k\leq N'} |\Psi_k| \bigg\|_{\textup{L}^p(\mathbb{R}^d)}. \label{eq:indauxa2}
\end{align}
The right hand sides of \eqref{eq:indauxa1} and \eqref{eq:indauxa2} are both at most
\[ 2^{-1/p} A_{p,q} C_{\textup{restr}} \bigg\| |f| \sum_{M'\leq k\leq N'} |\Psi_k| \bigg\|_{\textup{L}^p(\mathbb{R}^d)} \]
by our choice of $K$.
On the other hand, the a priori estimate \eqref{eq:restr} applied to the function $g=f\sum_{M'\leq k\leq K}\Psi_k$ gives
\begin{align}
\Bigg\| \sum_{M'\leq k\leq K\leq n\leq N'} \mathbbm{1}_{S_n} \widehat{f \Psi_k} \Bigg\|_{\textup{L}^q(S,\sigma)}
& = \big\| \mathbbm{1}_{S_K\cup\cdots\cup S_{N'}} \widehat{g} \big\|_{\textup{L}^q(S,\sigma)} \nonumber \\
& \leq C_{\textup{restr}} \bigg\| |f| \sum_{M'\leq k\leq N'} |\Psi_k| \bigg\|_{\textup{L}^p(\mathbb{R}^d)}. \label{eq:indauxa3}
\end{align}
Since the functions appearing on the left hand sides of \eqref{eq:indauxa1} and \eqref{eq:indauxa2} have disjoint supports, combining \eqref{eq:indauxa1}--\eqref{eq:indauxa3} we obtain exactly \eqref{eq:induct3} with the constant
\[ 2^{1/q-1/p} A_{p,q} + 1 = A_{p,q}. \]
This completes the induction step.

(b) We follow the proof of a similar variational inequality by Oberlin, Seeger, Tao, Thiele, and Wright \cite{OSTTW12:varcar}; see Section B in the appendix part of their paper.

We can assume $\varrho<q$, since decreasing $\varrho$ only increases the left hand side of \eqref{eq:induct2}.
Let us prove \eqref{eq:induct2} by the mathematical induction on the difference $N-M\in\mathbb{N}$, with the explicit constant given by
\[ B_{p,q,\varrho} := \frac{4}{(1-2^{1/q-1/p})(1-2^{1/\varrho-1/p})}. \]
The induction basis $N-M=1$ is again a trivial consequence of \eqref{eq:restr} applied to the function $f\Psi_{N}$, so we turn to the inductive step.
Similarly as in part (a) let $M+1\leq K\leq N$ be the smallest index such that
\[ \bigg\| |f| \sum_{M+1\leq k\leq K} |\Psi_k| \bigg\|_{\textup{L}^p(\mathbb{R}^d)}^p > \frac{1}{2} \bigg\| |f| \sum_{M+1\leq k\leq N} |\Psi_k| \bigg\|_{\textup{L}^p(\mathbb{R}^d)}^p. \]
Denote the supremum on the left hand side of \eqref{eq:induct2} by $V(M,N)$ if $M<N$ and interpret $V(M,N)=0$ otherwise. Observe
\[ V(M,N) \leq \Big( V(M,K-1)^\varrho + V(K+1,N)^\varrho \Big)^{1/\varrho} + 2\sup_{\substack{m,n\in\mathbb{Z}\\ M\leq m\leq K\leq n\leq N}} \Big| \sum_{k=m+1}^{n} \widehat{f \Psi_k} \Big|. \]
Taking the $\textup{L}^q(S,\sigma)$-norm and using our assumption $\varrho<q$ yields
\begin{align*}
\|V(M,N)\|_{\textup{L}^q(S,\sigma)} & \leq \Big( \|V(M,K-1)\|_{\textup{L}^q(S,\sigma)}^\varrho + \|V(K+1,N)\|_{\textup{L}^q(S,\sigma)}^\varrho \Big)^{1/\varrho} \\
& \quad + 4 \Bigg\| \sup_{\substack{n\in\mathbb{Z}\\ M+1\leq n\leq N}} \Big| \sum_{k=M+1}^{n} \widehat{f \Psi_k} \Big| \Bigg\|_{\textup{L}^q(S,\sigma)}.
\end{align*}
Using the induction hypothesis for $V(M,K-1)$ and $V(K+1,N)$ and applying part (a) to the last term, we obtain \eqref{eq:induct2} with the same constant $B_{p,q,\varrho}$, because it was chosen so that
\[ 2^{1/\varrho-1/p} B_{p,q,\varrho} + 4A_{p,q} = B_{p,q,\varrho}. \qedhere \]
\end{proof}

\section{Long variations}
For any $I\subseteq(0,\infty)$, a countable set or an interval, for $g,\psi\in\mathcal{S}(\mathbb{R}^d)$, and for $\xi\in\mathbb{R}^d$ we denote
\[ \|g\|_{\textup{V}^{\varrho}(I,\psi)}(\xi) := \sup_{\substack{m\in\mathbb{N}\\ t_0,t_1,\ldots,t_m\in I\\ t_0<t_1<\cdots<t_m}} \bigg( \sum_{j=1}^{m} \Big| \int_{t_{j-1}}^{t_j} (g\ast\psi_t)(\xi) \,\frac{\textup{d}t}{t} \Big|^\varrho \bigg)^{1/\varrho}. \]
Note that continuity of the above convolutions in $t$ guarantees measurability of $\|g\|_{\textup{V}^{\varrho}(I,\psi)}$ in the case when $I$ is an interval.
It will be convenient to work with functions $\psi\in\mathcal{S}(\mathbb{R}^d)$ such that $\widehat{\psi}$ is supported in the annulus
\begin{equation}\label{eq:annulus}
\big\{ x\in\mathbb{R}^d :  1\leq |x|\leq 2^{1/2} \big\}.
\end{equation}

In this section we establish a particular case of Theorem~\ref{thm:theorem2} when the variation is taken over integer powers of $2$ only, and the averaging is performed using a special, carefully chosen cutoff function.
Let us write $2^\mathbb{Z}$ for $\{2^k:k\in\mathbb{Z}\}$.

\begin{lemma}\label{lm:long}
For any $\psi\in\mathcal{S}(\mathbb{R}^d)$ such that $\widehat{\psi}$ is supported in the annulus \eqref{eq:annulus} and any $f\in\mathcal{S}(\mathbb{R}^d)$ we have
\[ \Big\| \big\|\widehat{f}\,\big\|_{\textup{V}^{\varrho}(2^\mathbb{Z},\psi)} \Big\|_{\textup{L}^q(S,\sigma)}
\lesssim_{p,q,\varrho} C_{\textup{restr}} \big\|\widehat{\psi}\,\big\|_{\textup{L}^{\infty}(\mathbb{R}^d)} \|f\|_{\textup{L}^p(\mathbb{R}^d)}. \]
\end{lemma}

\begin{proof}
Observe that
\[ \Psi_{k}(x) := \int_{2^{k-1}}^{2^{k-1/2}} \widecheck{\psi}(tx) \frac{\textup{d}t}{t} \]
is a Schwartz function supported in $2^{-k+1/2}\leq |x|\leq 2^{-k+3/2}$, so Lemma~\ref{lm:cklemma} applies to the sequence $(\Psi_k)_{k\in\mathbb{Z}}$.
Moreover,
\[ \big( \widehat{f \Psi_{k}} \big) (\xi)
= \int_{2^{k-1}}^{2^{k-1/2}} \big(\widehat{f\widecheck{\psi_{t}}}\big)(\xi) \frac{\textup{d}t}{t}
= \int_{2^{k-1}}^{2^{k-1/2}} \big(\widehat{f}\ast\psi_{t}\big)(\xi) \frac{\textup{d}t}{t}. \]
Also note that
\[ \|\Psi_k\|_{\textup{L}^{\infty}(\mathbb{R}^d)} \lesssim \big\|\widehat{\psi}\big\|_{\textup{L}^{\infty}(\mathbb{R}^d)}, \]
which, by disjointness of the supports of $\Psi_k$, clearly implies
\[ \Big\| \sum_{k=M+1}^{N}|\Psi_k| \Big\|_{\textup{L}^{\infty}(\mathbb{R}^d)} \lesssim \big\|\widehat{\psi}\big\|_{\textup{L}^{\infty}(\mathbb{R}^d)} \]
for any two integers $M$ and $N$ such that $M<N$.
Similarly, Lemma~\ref{lm:cklemma} can also be used with
\[ \Psi_{k}(x) := \int_{2^{k-1/2}}^{2^{k}} \widecheck{\psi}(tx) \frac{\textup{d}t}{t}, \]
since then the support of $\Psi_k$ lies in $2^{-k}\leq |x|\leq 2^{-k+1}$. This time we have
\[ \big( \widehat{f \Psi_{k}} \big) (\xi) = \int_{2^{k-1/2}}^{2^{k}} \big(\widehat{f}\ast\psi_{t}\big)(\xi) \frac{\textup{d}t}{t}. \]
Adding the two estimates obtained from \eqref{eq:induct2} gives
\[ \Bigg\| \sup_{\substack{m\in\mathbb{N}\\ k_0,\ldots,k_m\in\mathbb{Z} \\ M\leq k_0<\cdots<k_m\leq N}} \bigg( \sum_{j=1}^{m} \Big| \int_{2^{k_{j-1}}}^{2^{k_j}} \big(\widehat{f}\ast\psi_t\big) \frac{\textup{d}t}{t} \Big|^{\varrho} \bigg)^{1/\varrho} \Bigg\|_{\textup{L}^q(S,\sigma)}\lesssim_{p,q,\varrho} C_{\textup{restr}} \big\|\widehat{\psi}\,\big\|_{\textup{L}^{\infty}(\mathbb{R}^d)} \|f\|_{\textup{L}^p(\mathbb{R}^d)}, \]
so letting $M\to-\infty$ and $N\to\infty$ finishes the proof.
\end{proof}

\section{Short variations}
\begin{lemma}\label{lm:short}
For any $\psi\in\mathcal{S}(\mathbb{R}^d)$ such that $\widehat{\psi}$ is supported in the annulus \eqref{eq:annulus} and any $f\in\mathcal{S}(\mathbb{R}^d)$ we have
\[ \bigg\| \Big(\sum_{k\in\mathbb{Z}} \big\|\widehat{f}\,\big\|_{\textup{V}^{\varrho}([2^k,2^{k+1}],\psi)}^\varrho \Big)^{1/\varrho} \bigg\|_{\textup{L}^q(S,\sigma)}
\lesssim C_{\textup{restr}} \big\|\widehat{\psi}\,\big\|_{\textup{L}^{\infty}(\mathbb{R}^d)} \|f\|_{\textup{L}^p(\mathbb{R}^d)}. \]
\end{lemma}

\begin{proof}
Since the left hand side is decreasing in $\varrho$ and we have $\varrho>p$, it is enough to show the inequality with $\varrho$ replaced by $p$:
\[ \bigg\| \Big(\sum_{k\in\mathbb{Z}} \big\|\widehat{f}\,\big\|_{\textup{V}^{p}([2^k,2^{k+1}],\psi)}^p \Big)^{1/p} \bigg\|_{\textup{L}^q(S,\sigma)}
\lesssim C_{\textup{restr}} \big\|\widehat{\psi}\,\big\|_{\textup{L}^{\infty}(\mathbb{R}^d)} \|f\|_{\textup{L}^p(\mathbb{R}^d)}. \]
Take an arbitrary subdivision $2^k=t_0<t_1<\cdots<t_{m-1}<t_m=2^{k+1}$ of the interval $[2^k,2^{k+1}]$.
H\"{o}lder's inequality in $t$ gives
\[ \Big| \int_{t_{j-1}}^{t_j} \big(\widehat{f}\ast\psi_t\big)(\xi) \,\frac{\textup{d}t}{t} \Big|^p
\leq (2^k)^{p-1} \int_{t_{j-1}}^{t_j} \big|\big(\widehat{f}\ast\psi_t\big)(\xi)\big|^p \,\frac{\textup{d}t}{t^p}
\leq \int_{t_{j-1}}^{t_j} \big|\big(\widehat{f}\ast\psi_t\big)(\xi)\big|^p \,\frac{\textup{d}t}{t}, \]
where we used $t_j-t_{j-1}\leq 2^k$ and $t\geq t_{j-1}\geq 2^k$ respectively.
Summing in $1\leq j\leq m$ and taking supremum over all choices of the numbers $t_j$ we obtain the pointwise bound
\[ \big\|\widehat{f}\,\big\|_{\textup{V}^{p}([2^k,2^{k+1}],\psi)}^{p} \leq \int_{2^k}^{2^{k+1}} \big|\widehat{f}\ast\psi_t\big|^p \,\frac{\textup{d}t}{t}. \]
A further summation, this time in $k\in\mathbb{Z}$, gives
\[ \sum_{k\in\mathbb{Z}} \big\|\widehat{f}\,\big\|_{\textup{V}^{p}([2^k,2^{k+1}],\psi)}^p
\leq \int_{0}^{\infty} \big|\widehat{f}\ast\psi_t\big|^p \,\frac{\textup{d}t}{t}. \]

It suffices to establish a certain $\ell^p$-square function bound, namely
\begin{equation}\label{eq:squarefn}
\bigg\| \Big( \int_{0}^{\infty} \big|\widehat{f}\ast\psi_t\big|^p \,\frac{\textup{d}t}{t} \Big)^{1/p} \bigg\|_{\textup{L}^q(S,\sigma)}
\lesssim C_{\textup{restr}} \big\|\widehat{\psi}\,\big\|_{\textup{L}^{\infty}(\mathbb{R}^d)} \|f\|_{\textup{L}^p(\mathbb{R}^d)}.
\end{equation}
The assumption $p<q$ makes this quite easy. Indeed, we can interchange the $\textup{L}^q(S,\sigma)$ and $\textup{L}^{p}((0,\infty),\textup{d}t/t)$-norms, so that the left hand side of \eqref{eq:squarefn} is bounded by
\begin{equation}\label{eq:squarefn2}
\Big( \int_{0}^{\infty} \big\|\widehat{f}\ast\psi_t\big\|_{\textup{L}^q(S,\sigma)}^p \,\frac{\textup{d}t}{t} \Big)^{1/p}.
\end{equation}
Then we rewrite $\widehat{f}\ast\psi_t$ as $\widehat{f\widecheck{\psi}(t\,\cdot\,)}$ and apply \eqref{eq:restr} to the function $f\widecheck{\psi}(t\,\cdot\,)$ for each fixed $t$. This shows that \eqref{eq:squarefn2} is at most $C_{\textup{restr}}$ times
\begin{equation}\label{eq:squarefn3}
\Big( \int_{0}^{\infty} \big\|f\widecheck{\psi}(t\,\cdot\,)\big\|_{\textup{L}^p(\mathbb{R}^d)}^p \,\frac{\textup{d}t}{t} \Big)^{1/p}.
\end{equation}
Finally, expanding out the $\textup{L}^p(\mathbb{R}^d)$-norm, the $p$-th power of \eqref{eq:squarefn3} becomes
\[ \int_{\mathbb{R}^d} |f(x)|^p \Big( \int_{0}^{\infty} \big|\widecheck{\psi}(t x)\big|^p \,\frac{\textup{d}t}{t} \Big) \,\textup{d}x. \]
Note that for any nonzero $x\in\mathbb{R}^d$ we have
\[ \int_{0}^{\infty} \big|\widecheck{\psi}(t x)\big|^p \,\frac{\textup{d}t}{t}
\leq \int_{1/|x|\leq t\leq 2/|x|} \big\|\widehat{\psi}\,\big\|_{\textup{L}^{\infty}(\mathbb{R}^d)}^p \,\frac{\textup{d}t}{t} \lesssim \big\|\widehat{\psi}\,\big\|_{\textup{L}^{\infty}(\mathbb{R}^d)}^p. \qedhere \]
\end{proof}

\section{General averaging measures}
The argument that combines long and short variations is a standard one and we simply adapt it from \cite{JSW08:var}. In short, fix $m\in\mathbb{N}$ and any choice of numbers $t_0<t_1<\cdots<t_m$ from $(0,\infty)$. All integers $k$ such that $\{t_0,t_1,\ldots,t_m\}$ has a nonempty intersection with $[2^k,2^{k+1})$ are listed as
\[ k_1<k_2<\cdots<k_n. \]
Moreover, for each $i\in\{1,2,\ldots,n\}$ we list all elements of $\{t_0,t_1,\ldots,t_m\}$ that fall into $[2^{k_{i}},2^{k_{i}+1})$ as
\[ t^{(i)}_{1}<t^{(i)}_{2}<\cdots<t^{(i)}_{l_i}. \]
Then, writing $F(t):=(\widehat{f}\ast\psi_t)(\xi)$ for a fixed $\xi\in S$, we have
\begin{align*}
\sum_{j=0}^{m-1} \Big| \int_{t_{j}}^{t_{j+1}} F(t) \,\frac{\textup{d}t}{t} \Big|^\varrho
& \leq \sum_{i=1}^{n} \sum_{j=1}^{l_i-1} \Big| \int_{t^{(i)}_{j}}^{t^{(i)}_{j+1}} F(t) \,\frac{\textup{d}t}{t} \Big|^\varrho
+ 3^{\varrho-1} \sum_{i=1}^{n-1} \Big| \int_{2^{k_{i}}}^{t^{(i)}_{l_i}} F(t) \,\frac{\textup{d}t}{t} \Big|^\varrho \\
& \quad + 3^{\varrho-1} \sum_{i=1}^{n-1} \Big| \int_{2^{k_{i}}}^{2^{k_{i+1}}} F(t) \,\frac{\textup{d}t}{t} \Big|^\varrho
+ 3^{\varrho-1} \sum_{i=1}^{n-1} \Big| \int_{2^{k_{i+1}}}^{t^{(i+1)}_{1}} F(t) \,\frac{\textup{d}t}{t} \Big|^\varrho.
\end{align*}
This gives
\[ \big\|\widehat{f}\,\big\|_{\textup{V}^{\varrho}((0,\infty),\psi)}
\lesssim_\varrho \big\|\widehat{f}\,\big\|_{\textup{V}^{\varrho}(2^\mathbb{Z},\psi)}
+ \Big(\sum_{k\in\mathbb{Z}} \big\|\widehat{f}\,\big\|_{\textup{V}^{\varrho}([2^k,2^{k+1}],\psi)}^\varrho \Big)^{1/\varrho}. \]
Therefore, combining Lemmata~\ref{lm:long} and \ref{lm:short} we obtain
\begin{equation}\label{eq:combined}
\Big\| \big\|\widehat{f}\,\big\|_{\textup{V}^{\varrho}((0,\infty),\psi)} \Big\|_{\textup{L}^q(S,\sigma)}
\lesssim_{p,q,\varrho} C_{\textup{restr}} \big\|\widehat{\psi}\,\big\|_{\textup{L}^{\infty}(\mathbb{R}^d)} \|f\|_{\textup{L}^p(\mathbb{R}^d)}
\end{equation}
for every $\psi\in\mathcal{S}(\mathbb{R}^d)$ such that $\widehat{\psi}$ is supported in the annulus \eqref{eq:annulus} and every $f\in\mathcal{S}(\mathbb{R}^d)$.

Now we continue with a more general averaging measure. Suppose that $\mu$ is a complex measure on Borel subsets of $\mathbb{R}^d$ such that its Fourier transform $\widehat{\mu}$ is $\textup{C}^\infty$ and satisfies \eqref{eq:mudecay} for some $D\in[0,\infty)$ and $\eta>0$. Let us fix a radial $\textup{C}^\infty$ function $\varphi\colon\mathbb{R}^d\to[0,\infty)$ supported in the annulus \eqref{eq:annulus} and normalized so that
\begin{equation}\label{eq:phinormalization}
\int_{0}^{\infty} \varphi(sx) \frac{\textup{d}s}{s} =1
\end{equation}
for each nonzero $x\in\mathbb{R}^d$.
For each $s\in(0,\infty)$ define $\psi^{(s)}\in\mathcal{S}(\mathbb{R}^d)$ via its Fourier transform as
\[ \widehat{\psi^{(s)}}(x) = \varphi(x)\frac{x}{s}\cdot(\nabla\widehat{\mu})\Big(\frac{x}{s}\Big). \]
By \eqref{eq:mudecay} we have the decay
\begin{equation}\label{eq:phidecay}
\big\|\widehat{\psi^{(s)}}\big\|_{\textup{L}^{\infty}(\mathbb{R}^d)} \lesssim D \min\{s^\eta,s^{-1}\}.
\end{equation}
The fundamental theorem of calculus, normalization \eqref{eq:phinormalization}, decay \eqref{eq:phidecay}, and Fubini's theorem now give, for any $0<a<b$,
\begin{align*}
\widecheck{\mu}(bx) - \widecheck{\mu}(ax)
& = \int_{a}^{b} t \Big( \frac{\textup{d}}{\textup{d}t} \widecheck{\mu}(tx) \Big) \,\frac{\textup{d}t}{t}
= \int_{a}^{b} tx \cdot (\nabla\widecheck{\mu})(tx) \Big( \int_{0}^{\infty} \varphi(stx) \,\frac{\textup{d}s}{s} \Big) \,\frac{\textup{d}t}{t} \\
& = \int_{0}^{\infty} \int_{a}^{b} \widecheck{\psi^{(s)}}(stx) \,\frac{\textup{d}t}{t} \,\frac{\textup{d}s}{s}
= \int_{0}^{\infty} \int_{sa}^{sb} \widecheck{\psi^{(s)}}(tx) \,\frac{\textup{d}t}{t} \,\frac{\textup{d}s}{s}.
\end{align*}
Thus, for every $f\in\mathcal{S}(\mathbb{R}^d)$ we have
\[ f(x)\widecheck{\mu}(bx) - f(x)\widecheck{\mu}(ax) = \int_{0}^{\infty} \int_{sa}^{sb} f(x)\widecheck{\psi^{(s)}}(tx) \,\frac{\textup{d}t}{t} \,\frac{\textup{d}s}{s}. \]
Thanks to the absolute integrability coming from \eqref{eq:phidecay} we can take the Fourier transform of both sides and conclude
\begin{equation}\label{eq:decomposition}
\big(\widehat{f}\ast\mu_{b}\big)(\xi) - \big(\widehat{f}\ast\mu_{a}\big)(\xi)
= \int_{0}^{\infty} \int_{sa}^{sb} \big(\widehat{f}\ast\psi^{(s)}_{t}\big)(\xi) \,\frac{\textup{d}t}{t} \,\frac{\textup{d}s}{s}.
\end{equation}

Now we have all elements to deduce the two theorems formulated in the introductory section.

\begin{proof}[Proof of Theorem~\ref{thm:theorem2}]
From decomposition \eqref{eq:decomposition} interchanging the $\ell^\varrho$-norm and the integral in $s$ we get
\[ \Big( \sum_{j=1}^{m} \Big| \big(\widehat{f}\ast\mu_{t_{j}}\big)(\xi) - \big(\widehat{f}\ast\mu_{t_{j-1}}\big)(\xi) \Big|^{\varrho} \Big)^{1/\varrho}
\leq \int_{0}^{\infty} \Big( \sum_{j=1}^{m} \Big| \int_{st_{j-1}}^{st_{j}} \big(\widehat{f}\ast\psi^{(s)}_{t}\big)(\xi) \,\frac{\textup{d}t}{t} \Big|^{\varrho} \Big)^{1/\varrho} \,\frac{\textup{d}s}{s} \]
for any numbers $0<t_0<t_1<\cdots<t_m$.
By this and Minkowski's inequality we know that the left hand side of \eqref{eq:varrestr} is less than or equal to
\[ \int_{0}^{\infty} \Big\| \big\|\widehat{f}\,\big\|_{\textup{V}^{\varrho}((0,\infty),\psi^{(s)})} \Big\|_{\textup{L}^q(S,\sigma)} \,\frac{\textup{d}s}{s}, \]
which is by estimate \eqref{eq:combined} at most a constant depending on $p,q,\varrho$ times
\[  C_{\textup{restr}} \Big( \int_{0}^{\infty} \big\|\widehat{\psi^{(s)}}\big\|_{\textup{L}^{\infty}(\mathbb{R}^d)} \,\frac{\textup{d}s}{s} \Big) \|f\|_{\textup{L}^p(\mathbb{R}^d)}. \]
Finally, \eqref{eq:phidecay} bounds the last display by a constant multiple of
\[ C_{\textup{restr}} D \Big( \int_{0}^{\infty} \min\{s^{\eta-1},s^{-2}\} \,\textup{d}s \Big) \|f\|_{\textup{L}^p(\mathbb{R}^d)}
\lesssim_{\eta} C_{\textup{restr}} D \|f\|_{\textup{L}^p(\mathbb{R}^d)}. \qedhere \]
\end{proof}

\begin{proof}[Proof of Theorem~\ref{thm:theorem1}]
If in estimate \eqref{eq:varrestr} we take $\varrho=p+1$ and restrict the supremum to $m=1$ and $\{t_0,t_1\}=\{t,1\}$, we particularly get
\begin{equation}\label{eq:maxrestr2}
\Big\| \sup_{t\in(0,\infty)} \big| \widehat{f}\ast\mu_t - \widehat{f}\ast\mu \big| \Big\|_{\textup{L}^q(S,\sigma)}
\lesssim_{p,q,\eta} C_{\textup{restr}} D \|f\|_{\textup{L}^p(\mathbb{R}^d)}.
\end{equation}
In order to conclude \eqref{eq:maxrestr} it remains to write
\[ \Big\| \sup_{t\in(0,\infty)} \big| \widehat{f}\ast\mu_t \big| \Big\|_{\textup{L}^q(S,\sigma)}
\leq \big\| \widehat{f}\ast\mu \big\|_{\textup{L}^q(S,\sigma)}
+ \Big\| \sup_{t\in(0,\infty)} \big| \widehat{f}\ast\mu_t - \widehat{f}\ast\mu \big| \Big\|_{\textup{L}^q(S,\sigma)} \]
and observe that the two terms above are respectively bounded by \eqref{eq:restr} combined with Remark~\ref{rem:extend} and applied to $f\widecheck{\mu}$ and by \eqref{eq:maxrestr2}.
\end{proof}

\section*{Acknowledgments}
The author was supported in part by the Croatian Science Foundation under the project UIP-2017-05-4129 (MUNHANAP).
The author would like to thank the anonymous referee for the numerous useful comments and suggestions.


\end{document}